\documentclass[12pt]{article}
\usepackage{amsthm,amsfonts,amsmath,amssymb}
\usepackage{mathtools}
\usepackage{enumitem}
\usepackage{fancybox}
\usepackage{color}
\usepackage{float}
\usepackage{pgf,tikz}
\usetikzlibrary{arrows}
\usepackage{graphicx}
\usepackage{color}
\usepackage[all]{xy}

\newcommand{\Xd}{X}
\newcommand{\Li}{\mathcal{L}}
\newcommand{\LL}{\lvert \Li \rvert}
\newcommand{\LK}{\lvert K_{\Xd}\otimes\Li \rvert}
\newcommand{\LLn}{\lvert \Li_n \rvert}
\newcommand{\D}{\mathcal{D}}
\newcommand{\OO}{\mathcal{O}}
\newcommand{\scc}{\mathcal{S}}
\newcommand{\tq}{\mathcal{T}_q}
\newcommand{\C}{\mathbb{C}}
\newcommand{\R}{\mathbb{R}}
\newcommand{\Z}{\mathbb{Z}}
\newcommand{\N}{\mathbb{N}}
\newcommand{\F}{\mathbb{F}}

\newcommand{\T}{\mathcal{I}}
\newcommand{\MCG}{MCG}

\newcommand{\LS}{\mathcal{S}}
\newcommand{\CC}{\mathcal{C}}
\newcommand{\mmu}{[\mu]}
\newcommand{\cp}[1]{\C P^{#1}}

\newcommand{\inv}{\iota}
\newcommand{\ip}[2]{\langle #1,#2 \rangle}
\newcommand{\ipd}[2]{\langle #1,#2 \rangle_2}
\newcommand{\da}{\Delta_a}
\newcommand{\ak}{\alpha_{\kappa}}
\newcommand{\akp}{\alpha_{\kappa}'}
\newcommand{\ok}{\omega_{\kappa}}
\newcommand{\okp}{\omega_{\kappa}'}

\newcommand{\A}{\mathcal{A}}
\newcommand{\ttor}{(\C^\ast)^2}

\DeclareMathOperator{\Arf}{Arf}
\DeclareMathOperator{\im}{im}
\DeclareMathOperator{\id}{id}
\DeclareMathOperator{\spaut}{Sp}

\DeclareMathOperator{\itr}{int}

\newtheorem{Proposition}{Proposition}[section]
\newtheorem{Definition}[Proposition]{Definition}
\newtheorem{theorem}[Proposition]{Theorem}
\newtheorem{Lemma}[Proposition]{Lemma}
\newtheorem{Remark}[Proposition]{Remark}
\newtheorem{Theorem}{Theorem}

\newtheorem*{ack}{Acknowledgement}

\begin{document}

\title{The vanishing cycles of curves in toric surfaces II}
\author{R\'{e}mi Cr\'{e}tois and Lionel Lang}

\maketitle

\footnote{{\bf Keywords:}  toric varieties and Newton polygons,  vanishing cycles and monodromy, mapping class group, Torelli group, spin structures.} 
\footnote{{\bf MSC-2010 Classification:} 14M25,  20F38, 32S30.}

\abstract{
We resume the study initiated in \cite{CL}. For a generic curve $C$ in an ample linear system $\LL$ on a toric surface $\Xd$, a vanishing cycle of $C$ is an isotopy class of simple closed curve that can be contracted to a point along a degeneration of $C$ to a nodal curve in $\LL$. The obstructions that prevent a simple closed curve in $C$ from being a vanishing cycle are encoded by the adjoint line bundle $K_\Xd \otimes \Li$. \\
In this paper, we consider the linear systems carrying the two simplest types of obstruction. Geometrically, these obstructions manifest  on $C$  respectively as an hyperelliptic involution and as a spin structure. In both cases, we determine all the vanishing cycles by investigating the associated monodromy maps, whose target space is the mapping class group $\MCG(C)$. We show that the image of the monodromy is the subgroup of $\MCG(C)$ preserving respectively the hyperelliptic involution and the spin structure.\\
The results obtained here support the Conjecture $1$ in \cite{CL} aiming to describe all the vanishing cycles for any pair $(\Xd, \Li)$.
}

\section{Introduction}

Let $\Xd$ be a smooth and complete toric surface and consider an ample line bundle $\Li$ on it. Assume moreover that the genus of a generic curve $C_0$ in the complete linear system $\LL = \mathbb{P}(H^0(\Xd,\Li))$ is at least $1$. Since $\Xd$ is smooth, $\Li$ is very ample.

 A vanishing cycle on $C_0$ is an isotopy class of non-separating simple closed curves on $C_0$ which is contracted to a node along a degeneration of $C_0$ in $\LL$ to an irreducible nodal curve of geometric genus one less. In \cite{CL}, we considered the question of which non-separating simple closed curve on $C_0$ is a vanishing cycle. In particular, we underlined obstructions in terms of $n^{th}$ roots of the adjoint bundle $K_{\Xd} \otimes \Li$. Moreover, we showed that when there is no obstruction, then any isotopy class of non-separating simple closed curves on $C_0$ is a vanishing cycle (\cite[Theorem 1]{CL}).

In the present paper, we consider the linear systems exhibiting the two simplest types of obstruction. For such linear systems, we give a complete description of the vanishing cycles in terms of the obstructions. More precisely, we prove the following two theorems. 

 \begin{Theorem}\label{thm:ihyper}
Assume that the image of $\Xd\rightarrow \LK^*$ is of dimension~$1$.\\
   Then the curve $C_0$ admits an hyperelliptic involution $\inv$. Moreover, the vanishing cycles on $C_0$ are the isotopy classes of non-separating simple closed curves on $C_0$ which are invariant by $\inv$.
 \end{Theorem}

 \begin{Theorem}\label{thm:ispin}
Assume that the image of $\Xd\rightarrow \LK^*$ is of dimension~$2$ and that the largest order of a root $\LS$ of $K_{\Xd}\otimes \Li$ is $2$.\\
  Then, the restriction $\LS_{|C_0}$ of $\LS$ to the curve $C_0$ is a spin structure. Moreover, a simple closed curve on $C_0$ is a vanishing cycle if and only if its tangent framing lifts to the spin structure $\LS_{|C_0}$.
 \end{Theorem}
For the definition of spin structure, we refer to \cite{Ati}. Recall also from \cite[Proposition 3.2]{Ati} that spin structures are in natural bijection with the square roots of the canonical line bundle.

Note that since the genus of the generic curve in $\LL$ is at least $1$, the line bundle $K_{\Xd}\otimes \Li$ has an empty base locus (see \cite[Proposition 4.3]{CL}) and thus the natural map $\Xd\rightarrow\LK^*$ is well-defined.

 The study of the vanishing cycles on $C_0$ can be reformulated in terms of the monodromy map. Indeed, under our hypothesis on $\Li$, the discriminant $\D \subset \LL$ parametrizing singular curves is an irreducible hypersurface of $\LL$ (see \cite[6.5.1]{DIK}]). The geometric monodromy map
\[
\mu : \pi_1(\LL\setminus\D,C_0)\rightarrow \MCG(C_0)
\]
is defined using any trivialisation of the universal curve over a loop in $\LL\setminus\D$. Its weaker version, the algebraic monodromy map
\[
\mmu : \pi_1(\LL\setminus \D,C_0)\rightarrow \spaut(H_1(C_0,\Z))
\]
 is defined by composing $\mu$ with the natural map $\MCG(C_0)\rightarrow \spaut(H_1(C_0,\Z))$, where $\spaut(H_1(C_0,\Z))$ is the group of automorphisms of $H_1(C_0,\Z)$ which preserve the intersection form.

The relation with our original problem is as follows. If $\delta\subset C_0$ is a vanishing cycle then the Dehn twist $\tau_{\delta}$ along $\delta$ is in the image of $\mu$ (see \cite[\S 10.9]{ACG2}). However, the converse is not obviously true. Still, we prove the two following theorems and show that they imply Theorems \ref{thm:ihyper} and \ref{thm:ispin}.

 \begin{Theorem}\label{thm:hyperm}
Assume that the image of $\Xd\rightarrow \LK^*$ is of dimension~$1$.\\
Then the curve $C_0$ admits an hyperelliptic involution $\inv$ and the image of $\mu$ is the associated hyperelliptic mapping class group
\[
\MCG(C_0,\inv) = \{\varphi\in\MCG(C_0)\ |\ \varphi\circ\inv = \inv\circ\varphi\}.
\]
 \end{Theorem}
 
 \begin{Theorem}\label{thm:spinm}
Assume that the image of $\Xd\rightarrow \LK^*$ is of dimension~$2$ and let $n$ be the largest order of a root $\LS$ of $K_{\Xd}\otimes \Li$.\\
 If $n=2$ then the monodromy map $\mu$ is surjective on the subgroup $\MCG(C_0,\LS_{|C_0})$ of elements of the mapping class group of $C_0$ which preserve the spin structure $\LS_{|C_0}$.\\
If $n$ is even, then the algebraic monodromy map $\mmu$ is surjective on the subgroup $\spaut(H_1(C_0,\Z),\LS_{|C_0}^{\otimes \frac{n}{2}})$ of elements of $\spaut(H_1(C_0,\Z))$ which preserve the spin structure $\LS_{|C_0}^{\otimes \frac{n}{2}}$.
 \end{Theorem}

Recently, a lot of progress has been made in investigating  vanishing cycles in linear systems on toric surfaces. For that reason, we wish to contextualize the present work within the recent developments in the subject.

In \cite{Salter1}, Salter described the obstructions to the surjectivity of the monodromy map in the case of degree $d$ curves in $\cp{2}$. Moreover, he showed that the image of the monodromy is the spin mapping class group for $d=5$. Shortly after in \cite{CL}, the authors extended the catalogue of obstructions to any complete linear system on smooth toric surfaces and showed that the monodromy map is surjective when no obstruction shows up. Few months after appeared the first version of the present work. Finally in \cite{Salter2}, Salter completed the description of vanishing cycles of curves in smooth toric surfaces. 

During the revision process of the present work, it appeared that the original proof of Theorem \ref{thm:spinm} could be substantially simplified using the techniques of \cite{Salter2}. It is worth pointing out that, while Theorem \ref{thm:ispin} is covered by \cite[Theorem A]{Salter2} it is not the case for Theorem \ref{thm:spinm}. 

The missing piece is to show that the spin mapping class group is generated by Dehn twists, the so-called \emph{admissible twists} in \cite{Salter2}. We take care of this in Theorem \ref{thm:spin} that we prove in Section \ref{sec:generators}. We prove Theorems \ref{thm:ihyper} and \ref{thm:hyperm} in Section \ref{sec:hyper} and Theorems \ref{thm:ispin} and \ref{thm:spinm} in Section \ref{sec:spinst}. Finally, we provide explicit computations of the spin structure in the appendix.

\begin{ack} The authors are grateful to Denis Auroux,  Sylvain Courte, Simon K. Donaldson, Christian Haase, Tobias Ekholm, Ilia Itenberg, Grigory Mikhalkin, Nick Salter and all the participants of the learning seminar on Lefschetz fibrations of Uppsala. The authors are also indebted to an anonymous referee. Finally, the two authors were supported by the Knut and Alice Wallenberg Foundation.
\end{ack}

\tableofcontents

\section{Polygons and monodromy}\label{sec:poly}

In this section, we recall some of the results of \cite{CL} which we will need to prove Theorems \ref{thm:hyperm} and \ref{thm:spinm} and carry out the computations in the appendix.

Take a smooth and complete toric surface $\Xd$ associated to a fan in $\Z^2\otimes \R$ and an ample line bundle $\Li$ on it. Given such a pair, we can define a convex lattice polygon $\Delta\subset\R^2$ which is well-defined up to affine integer transformation (see \cite[\S 3]{CL}). Since $\Xd$ is smooth and $\Li$ is ample, the polygon $\Delta$ is smooth, that is any pair of primitive integer vectors directing two consecutive edges of $\Delta$ generates the lattice $\Z^2$.

Throuhout the paper, we assume that the arithmetic genus $g_{\Li}$ of the curves in $\LL$ is at least $1$. Recall by \cite{Kho} that $g_{\Li}$ is equal to the cardinality of $\itr(\Delta)\cap \Z^2$. Under the present assumption, we can define the interior polygon $\da$ as the convex hull of the non-empty set $\itr(\Delta)\cap \Z^2$. Moreover, the polygon $\da$ is smooth as soon as $\dim(\da)=2$ (see \cite[Proposition 3.3]{CL}). 

The description of line bundles on toric surfaces in terms of lattice polygons relates the interior polygon $\da$ to the adjoint line bundle $K_{\Xd}\otimes \Li$ (see \cite[Proposition 3.3]{CL}). In the present context, the adjoint line bundle $K_{\Xd}\otimes \Li$ is nef with empty base locus. Thus, the map $\phi_{K_{\Xd}\otimes \Li}:\Xd\rightarrow \LK^*$ is well-defined. Moreover, we have that $\dim(\im(\phi_{K_{\Xd}\otimes \Li})) = \dim(\da)$.

As we saw in \cite[Proposition 3.1]{CL}, the largest order $n$ of a root of $K_{\Xd}\otimes\Li$ is equal to the greatest common divisor of the integer lengths of the edges of $\da$. Recall that the integer length of a closed segment $\sigma \subset \R^2$ joining two lattice points is equal to $\vert \sigma \cap \Z^2\vert -1$.

\begin{Definition}\label{def:even}
A convex lattice polygon $\Delta\subset \R^2$ is \textbf{even}  if the greatest common divisor of the integer lengths of its edges is even.\\
In that case, a lattice point $p$ of $\Delta$ is \textbf{even} if for a vertex $v$ of $\Delta$ (and then for any) the difference $p-v$ has even coordinates. Otherwise, we say that $p$ is \textbf{odd}.
\end{Definition}

We now recall a combinatorial construction from \cite{CL} that allows us to describe a lot of elements in $\im(\mu)$. Define $\tilde{\Delta}$ to be the closure of $\Delta \setminus \Z^2$ in the real oriented blow-up of $\R^2$ at all the points of $\Z^2$. Denote by $\pi : \tilde{\Delta} \twoheadrightarrow \Delta$ the restriction of the blow-up and define $D_\bullet:=\pi^{-1}(\Delta \cap \Z^2) \subset \partial \tilde{\Delta}$. Now define $C_{\Delta}^{\circ}$ to be the gluing of two copies of $\tilde{\Delta}$ with opposite orientation along $D_\bullet$. Thus $C_{\Delta}^{\circ}$ is a surface of genus $g_{\Li}$ with $\vert\partial \Delta \cap \Z^2\vert$ many boundary components. This surface comes with a map $pr : C_{\Delta}^{\circ} \rightarrow \Delta$ defined as the composition of the projection $C_{\Delta}^{\circ} \twoheadrightarrow \tilde{\Delta}$ with the blow-up $\pi$. Finally, denote by $C_{\Delta}$ the compactification of $C_{\Delta}^{\circ}$ obtained by contracting each boundary component of $C_{\Delta}^{\circ}$ to a point, see Figure \ref{fig:cdelta}. 

Recall that a curve $C\in \LL$ is a simple Harnack curve if it is defined over $\R$ and if its real part $\R C$ satisfies  particular topological properties, see $e.g.$ \cite{Mikh}. In \cite[Proposition 4.6]{CL}, we defined a diffeomorphism $R_C$ between any smooth simple Harnack curve $C\in \LL$ and $C_{\Delta}$, essentially given by the gradient of the Ronkin function.

\begin{Definition}\label{def:segm}
A \textbf{primitive integer segment} is a segment in $\Delta$ that joins two lattice points and whose integer length is $1$. A primitive integer segment is a \textbf{bridge} if it joins a lattice point of $\partial \Delta$ to a lattice point of $\partial \da$ and does not intersect $\itr (\da)$.
\\
For any primitive integer segment $\sigma \subset \Delta$ and any smooth simple Harnack curve $C\in \LL$, define the simple closed curve $\delta_\sigma := (R_{C} \circ pr)^{-1}(\sigma) \subset C$ and $\tau_\sigma \in \MCG(C)$ to be the Dehn twist along  $\delta_\sigma$.\\
For any $v \in \da \cap \Z^2$, define the simple closed curve $\delta_v := (R_{C} \circ pr)^{-1}(v) \subset \R C$ and $\tau_v \in \MCG(C)$ to be the Dehn twist along  $\delta_v$. We will refer to the $\delta_v$'s as the \textbf{$A$-cycles} of $C$.
\end{Definition}

\begin{figure}[h]
\centering
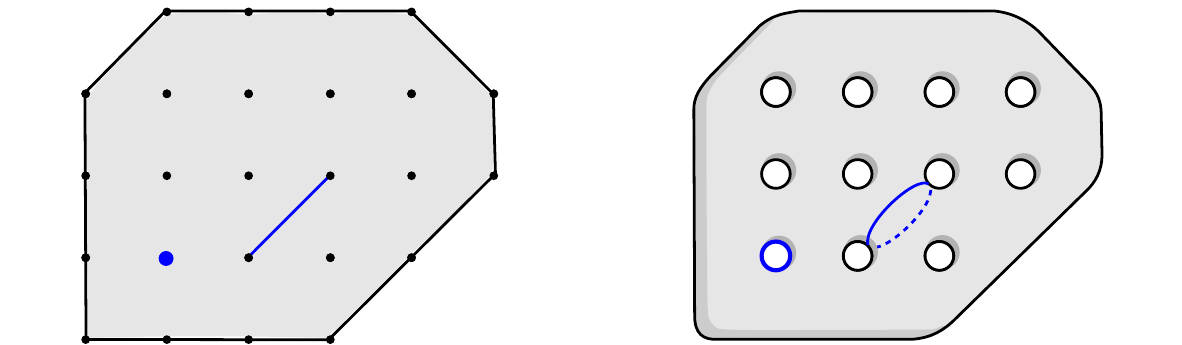
\caption{$C_\Delta$, $\delta_v$ and $\delta_\sigma$ from $\Delta$, $v$ and $\sigma$}
\label{fig:cdelta}
\end{figure}

The following results collected from \cite{CL} will already allow us to prove Theorem \ref{thm:hyperm}.

\begin{theorem}\label{thm:rappel}
Let $C_0$ be a smooth simple Harnack curve in $\LL$.
\begin{enumerate}
\item Any  $A$-cycle of $C_0$ is a vanishing cycle in the linear system $\LL$. (\cite[Theorem 3]{CL})
\item Let $\kappa\in\da$ be a vertex. If $\sigma$ is a bridge ending at $\kappa$, then $\tau_{\sigma}$ is in $\im(\mu)$.(\cite[Proposition 6.6]{CL})
 \item Let $\sigma$ be a primitive integer segment lying on an edge of $\da$. The Dehn twist $\tau_{\sigma}$ is in $\im(\mu)$.(\cite[Proposition 6.7]{CL})
 \end{enumerate}
\end{theorem}

\begin{Remark}\label{rem:sympbas}
Notice that $\delta_v$ and $\delta_\sigma$ intersect if and only if $v$ is an end point of $\sigma$. In this case, $\delta_v$ intersects $\delta_\sigma$ transversally at one point.
\end{Remark}

The following results from \cite{CL} are not necessary to prove Theorem \ref{thm:spinm} as we will use the results of \cite{Salter2} instead. Nevertheless, they allow us to compute the quadratic form associated to the canonical spin structure, see Section \ref{sec:app}.

\begin{Proposition}[Propositions 7.13, 7.15 and 7.16 in \cite{CL}]\label{prop:rappel}
Assume that the largest order of a root of $K_{\Xd}\otimes \Li$ is $2$ (resp. is even). If $\sigma\subset \Delta$ is a primitive integer segment with one end being an even point of $\da$, then $\tau_{\sigma}$ is in $\im(\mu)$ (resp. $[\tau_{\sigma}]\in\im(\mmu)$). Otherwise $\tau_{\sigma}^2$ is in $\im(\mu)$ (resp.  $[\tau_{\sigma}]^2\in\im(\mmu)$).
\end{Proposition}

\section{The hyperelliptic case}\label{sec:hyper}

 In this section we prove Theorems \ref{thm:hyperm} and \ref{thm:ihyper}.  In particular, we assume that the dimension of $\da$ is $1$ (see \cite[Proposition 3.3]{CL}). We start by describing which linear systems satisfy this condition.

For any non-negative integer $\alpha$, let $\F_\alpha$ be the Hirzebruch surface obtained as the projectivization of $\OO_{\cp{1}}\oplus\OO_{\cp{1}}(\alpha)\rightarrow \cp{1}$ and denote by $\pi_\alpha :\F_\alpha\rightarrow \cp{1}$ the induced projection. Its Picard group is free of rank $2$ generated by the fiber class $f$ and the tautological class $h$ and its canonical class is $-2h+(\alpha-2)f$.

 A line bundle $\Li$ over $\F_\alpha$ is ample if and only if its class is given by $mh+nf$ with $m,n>0$. Let $\Li_n$ be the line bundle of class $2h+nf$ for $n>0$. Then $K_{\F_\alpha}\otimes \Li_n$ is not ample as its intersection with $f$ vanishes. Moreover, as soon as $\alpha+n-1>1$, the dimension of $\Delta_{K_{\F_\alpha}\otimes \Li_n}$ is $1$. In fact, this is essentially the only case where this last property can happen, as shown by Proposition \ref{onedim} below.

 \begin{Proposition}\label{onedim}
   Let $\Xd$ be a smooth and complete toric surface and let $\Li$ be an ample line bundle on $\Xd$. Denote by $\Delta$ the associated convex lattice polygon. Assume that the dimension of $\da$ is $1$. Then there exists $\alpha, n\in\N$ with $n>0$ and $\alpha+n-1 = g_{\Li}$ and a map $\Xd\rightarrow \F_\alpha$ which is either
   \begin{enumerate}
   \item an isomorphism, and then $\LL = \LLn$,
   \item a blow-up at one point $p\in\F_\alpha$, and then $\LL$ is the proper transform of
\[
\{[s]\in \LLn\ |\ s(p) = 0\}\subset \LLn
\]
under the blow-up,
   \item or a blow-up at two distinct points $p, q\in \F_\alpha$, and then $\LL$ is the proper transform of
\[
\{[s]\in \LLn\ |\ s(p) = 0, s(q)=0\}\subset \LLn
\]
under the blow-up.
   \end{enumerate}\qed
 \end{Proposition}

For a proof of Proposition \ref{onedim}, we refer the reader to \cite[\S 4.3]{Koe} where the classification of polygons $\Delta$ such that the dimension of $\da$ is $1$ can be found.

 \begin{proof}[Proof of Theorem \ref{thm:hyperm}]
Consider first the case of an Hirzebruch surface $\pi_\alpha:\F_\alpha\rightarrow\cp{1}$ and an ample line bundle $\Li_n= \OO_{\F_\alpha}(2h+nf)$ on it, for $n>0$. All the smooth curves in $\LLn$ are hyperelliptic. More precisely, let $\CC_n\subset \LLn\times\F_\alpha$ denote the universal curve over $\LLn\setminus \D$. Then $\pi_\alpha$ induces a map $\CC_n\rightarrow (\LLn\setminus\D) \times \cp{1}$ which is of degree $2$. It follows from \cite[Theorem 5.5]{KlL} that there exists an involution $I$ of $\CC_n$ over $\LLn\setminus\D$ which restricts to the hyperelliptic involution on each curve. 

Take a smooth curve $C_0\in \LLn$ and denote by $\iota$ the restriction of $I$ to $C_0$. Let $\gamma : [0,1]\rightarrow \LLn\setminus \D$ be a smooth path starting and ending at $C_0$ and choose a smooth trivialisation $\Phi : C_0\times [0,1]\rightarrow \gamma^*\CC_n$ with $\Phi(.,0)=\id$, so that the monodromy associated to $\gamma$ is $\phi = \Phi(.,1)$. The map $\Phi^{-1}\circ I\circ \Phi$ is an isotopy from $\iota$ to $\phi^{-1}\circ\iota\circ\phi$ so that the class of $\phi$ in $\MCG(C_0)$ commutes with the class of $\iota$. Thus the image of the geometric monodromy $\mu$ is a subgroup of the hyperelliptic mapping class group $\MCG(C_0,\iota)$.

Take now any smooth complete toric surface $\Xd$ and ample line bundle $\Li$ and let $\Delta$ be the associated polygon. Assume that the interior polygon $\da$ is of dimension $1$. Using Proposition \ref{onedim} and the preceding argument about the Hirzebruch surfaces, we conclude that in this case also any smooth curve $C_0\in\LL$ is hyperelliptic with involution $\iota$ and that the image of $\mu$ is a subgroup of $\MCG(C_0,\iota)$.

Let us show now that the image is in fact the whole hyperelliptic group. Fix a smooth Harnack curve $C_0\in\LL$. From Proposition \ref{onedim}, we can assume that the polygon $\Delta$ has a vertex at $(0,0)$ with adjacent edges directed by $(0,1)$ and $(1,0)$ and that $\da$ is the segment joining $(1,1)$ to $(g_{\Li},1)$. The point $(g_{\Li}+1,1)$ is also on the boundary of $\Delta$. Denote by $\sigma_i$ the primitive integer segments $[(i,1),(i+1,1)]$, $i=0,\ldots,g_{\Li}$ and by $v_i$, $i=1,\ldots,g_{\Li}$, the points $(i,1)$. It follows from Theorem \ref{thm:rappel} that $\tau_{\sigma_i}\in\im(\mu)$ for all $i = 0,\ldots,g_{\Li}$ and that $\tau_{v_i}\in\im(\mu)$ for $i=1,\ldots,g_{\Li}$. Notice also that each curve $\delta_{\sigma_i}$ and $\delta_{v_i}$ is isotopic to its image by $\iota$.

Let $j = \tau_{\sigma_0}\tau_{v_1}\tau_{\sigma_1}\ldots\tau_{v_{g_{\Li}}}\tau_{\sigma_{g_{\Li}}}\tau_{\sigma_{g_{\Li}}}\tau_{v_{g_{\Li}}}\ldots\tau_{\sigma_1}\tau_{v_1}\tau_{\sigma_0}$. Using the Alexander method (see \cite[Proposition 2.8]{farbmarg}), we know that $\iota \circ j$ is isotopic to the identity. In particular, $\MCG(C_0,\iota)=\MCG(C_0,j)$. Moreover, the group $\MCG(C_0,j)$ is generated by the Dehn twists $\tau_{\sigma_0},\tau_{v_1},\tau_{\sigma_1},\ldots,$ $\tau_{v_{g_{\Li}}},\tau_{\sigma_{g_{\Li}}}$ (see \cite[Theorem 9.2]{farbmarg}). Thus $\im(\mu) = \MCG(C_0,\iota)$.
 \end{proof}

 To finish the proof of Theorem \ref{thm:ihyper}, we use Lemma \ref{lem:symcurves} below.

 \begin{Lemma}\label{lem:symcurves}
   Let $C$ be a compact Riemann surface of genus $g\geq 2$ with an hyperelliptic involution $\iota : C\rightarrow C$. Let $\delta$ and $\delta'$ be non-separating smooth simple closed curves on $C$. 

If $\delta$ is isotopic to $\iota(\delta)$ then there exists a smooth simple closed curve $\delta_0$ on $C$ which is invariant under $\iota$ and isotopic to $\delta$. 

If $\delta'$ is also isotopic to $\iota(\delta')$, then there exists a diffeomorphism $\phi$ of $C$ sending $\delta$ to $\delta'$ and such that $[\phi]\in\MCG(C,\iota)$.
 \end{Lemma}

 \begin{proof}
   Since $C$ is hyperelliptic, $\iota$ is an isometry for its hyperbolic metric. Denote by $\delta_0$ the unique geodesic in the isotopy class of $\delta$. Then $\iota(\delta_0)$ is again a geodesic and is isotopic to $\delta_0$ by assumption. By unicity of the geodesic, $\delta_0 =\iota(\delta_0)$.

For the second part of the Lemma, we can assume that both $\delta$ and $\delta'$ are invariant by $\iota$. The quotient $C/\iota$ is a sphere with $2g+2$ marked points. The images of $\delta$ and $\delta'$ in this quotient are two simple arcs joining two marked points. Using a diffeomorphism of the sphere permuting the marked points, we can send one of those arc on the other. Since such a diffeomorphism lifts to a diffeomorphism of $C$ which is invariant by $\iota$, this concludes the proof.
 \end{proof}

 \begin{proof}[Proof of Theorem \ref{thm:ihyper}]
Let $C_0$ be a smooth Harnack curve in $\LL$. We know from Theorem \ref{thm:hyperm} that $C_0$ admits an hyperelliptic involution $\iota$ and that any vanishing cycle on $C_0$ has to be invariant by $\iota$.

On the other hand, we know from Theorem \ref{thm:rappel} that any $A$-cycle on $C_0$ is a vanishing cycle. Choose one of those and denote it by $\delta$. Now, if $\delta'\subset C_0$ is another simple closed curve which is isotopic to $\iota(\delta')$, then by Lemma \ref{lem:symcurves}, there exists $[\phi]\in \im(\mu)$ such that $\phi(\delta) = \delta'$. This implies that $\delta'$ is also vanishing cycle (see \cite[Proposition 3.23]{Voisin2}), which concludes the proof of Theorem \ref{thm:ihyper}.
 \end{proof}

\section{Generators of the spin mapping class group}\label{sec:generators}

This section is a preparation for proving Theorems \ref{thm:ispin} and  \ref{thm:spinm}.  For the moment, we forget about the specific case of curves in toric surfaces and aim to prove some general statements about generating sets of the spin mapping class group.  

Consider an orientable surface $\Sigma$ of genus $g$, together with a spin structure $S$ on it. Recall from \cite[Theorem 3A]{johnson} that there is a natural bijection between the set of spin structures on $\Sigma$ and the set of quadratic forms on $H_1(\Sigma,\Z/2\Z)$ associated to the \emph{mod $2$} intersection pairing $\ipd{.}{.}$. Let us recall that  a quadratic form $q : H_1(\Sigma,\Z/2\Z) \rightarrow \Z/2\Z$ is a map such that for any $a$ and $b$ in $H_1(\Sigma,\Z/2\Z)$, we have 
$$q(a+b) = q(a)+q(b)+\ipd{a}{b}.$$
Denote by $\scc(\Sigma)$ the set of isotopy classes of oriented simple closed curves in $\Sigma$. By composition with the natural maps 
\[ \scc(\Sigma) \rightarrow H_1(\Sigma,\Z) \rightarrow H_1(\Sigma,\Z/2\Z), \]
we define the quadratic form $q$ on both $\scc(\Sigma)$ and $H_1(\Sigma,\Z)$ and still denote it (somewhat abusively) by $q$. Note that the above map $\scc(\Sigma) \rightarrow H_1(\Sigma,\Z/2\Z)$ descends in fact to a map on unoriented simple closed curves on $\Sigma$, and so does $q$. Note also that $\scc(\Sigma)$ is acted upon by $\MCG(\Sigma)$. In turn, the mapping class group $\MCG(\Sigma)$ acts on the set of quadratic forms on $\scc(\Sigma)$ by pre-composition.  

From now on, the letter $q$ refers to the quadratic form associated to the spin structure $S$. Define the \textbf{spin mapping class group} as
\[\MCG(\Sigma,q):=\left\lbrace \phi \in \MCG(\Sigma) \, \vert \, q \circ \phi = q  \right\rbrace.\]

Our interest in this section is to find generators for the spin mapping class group. To that aim, we introduce the notion of admissible twist, following \cite{Salter2}.

\begin{Definition}\label{def:adm}
A simple closed curve $c \subset \Sigma$ is an \textbf{admissible curve} if it is non-separating  and if $q(c)=1$. For any simple closed curve $c \subset \Sigma$, denote by $\tau_c \in \MCG(\Sigma)$ the Dehn twist  along $c$.  A Dehn twist $\tau_c \in \MCG(\Sigma)$ is an \textbf{admissible twist} if $c$ is admissible. Finally, denote by $\tq < \MCG(\Sigma)$ the subgroup generated by all admissible twists.
\end{Definition}

Our interest for admissible curves is motivated by the following.

\begin{Lemma}\label{lem:adm}
A simple closed curve $c \subset \Sigma$ is admissible  if and only if $\tau_c \in \MCG(\Sigma,q)$. In particular, we have $\tq < \MCG(\Sigma,q)$.
\end{Lemma}

\begin{proof}
Let $\tau_c$ be the Dehn twist along a simple closed curve $c\subset \Sigma$. For any $d \in \scc(\Sigma)$, we have the following equality 
\[
\left[ \tau_c(d) \right]_2 = \left[d \right]_2 + \ipd{c}{d} \left[c \right]_2
\]
in $H_1(\Sigma,\Z/2\Z)$. By definition of a quadratic form on $H_1(\Sigma,\Z/2\Z)$, it follows from the above equality that
$$
q(\tau_c (d)) =   q(d) + \ipd{c}{d} \cdot q(c) +\ipd{c}{d}
$$
and hence that $q(\tau_c(d)) = q(d)$ for all $d\in\scc(\Sigma)$ if and only if $q(c) =1$.
\end{proof}

Now, we are in position to state the main statement of this section.

\begin{Theorem}\label{thm:spin}
For $g \geq 5$, we have $\tq = \MCG(\Sigma,q)$.
\end{Theorem}

\begin{proof}
The mapping class group $\MCG(\Sigma)$ acts on $H_1(\Sigma,\Z)$ by automorphisms preserving the intersection form. As a consequence, it induces a morphism
\begin{equation}\label{eq:natmorph}
\MCG(\Sigma) \rightarrow \spaut(H_1(\Sigma,\Z))
\end{equation}
whose kernel is referred to as the Torelli group $\T$ of $\Sigma$, see $e.g$ \cite[Section 6.5]{farbmarg}. Restricting this morphism to $\MCG(\Sigma,q)$, we obtain a short exact sequence

\begin{equation}\label{eq:exactspin}
1\rightarrow \T \rightarrow \MCG(\Sigma,q)\rightarrow \spaut(H_1(\Sigma,\Z),q)\rightarrow 1,
\end{equation}
where $\spaut(H_1(\Sigma,\Z),q)$ is the subgroup of $\spaut(H_1(\Sigma,\Z))$ preserving $q$.
In order to prove Theorem \ref{thm:spin}, it is enough to show that $\spaut(H_1(\Sigma,\Z),q)$ is a subgroup of $[\tq]$ and that $\T$ is a subgroup of $\tq$. We prove the first statement in Proposition \ref{prop:spqtq} in Section \ref{sec:spinsymp} and the second one in Proposition \ref{prop:torelli} in  Section  \ref{sec:torelli}.
\end{proof}

\subsection{A short reminder on spin structures}\label{sec:spin}

Consider any collection $a_1,b_1,a_2,b_2,...,a_g,b_g \in \scc(\Sigma)$ forming a geometric symplectic basis of $\Sigma$, see \cite[Section 6.1.2]{farbmarg}. For any index $i$, the quadratic form $q$ behaves in two different manners on the pair $\{a_i,b_i \}$: 
\[ \text{either } q(a_i)  \cdot q(b_i)=1 \text{ or } q(a_i) \cdot q(b_i)=0. \]  
In the first case, we have necessarily that $q(a_i) = q(b_i)=1$. There are several possibilities in the second case. Nevertheless, we can always assume that $\big( q(a_i),q(b_i)\big)=(0,1)$ up to exchanging $a_i$ with $b_i$ and replacing $b_i$ by $\tau_{a_i}(b_i)$.

\begin{Definition}\label{def:qsymp}
The collection $a_1,b_1,a_2,b_2,...,a_g,b_g$ is a\textbf{ $q$-symplectic basis} if for any index $i$, we have either $\big( q(a_i),q(b_i)\big)=(1,1)$ or $\big( q(a_i),q(b_i)\big)=(0,1)$. For a given $q$-symplectic basis, we say that the index $i$ is of \textbf{type} $0$ (respectively $1$) if  $q(a_i)=0$ (respectively $q(a_i)=1$).
\end{Definition}

As we discussed above, $q$-symplectic bases always exist. Note also that a pair of indices $(i,j)$ of type $0$ can be easily turned into a pair of indices of type $1$ and vice versa. Indeed, it suffices to replace $a_i$ and $a_j$ respectively by $\tau_c(a_i)$ and $\tau_c(a_j)$ where $c$ is depicted in Figure \ref{fig:spin} and $q(c)=0$ by construction.

\begin{figure}[h]
\centering
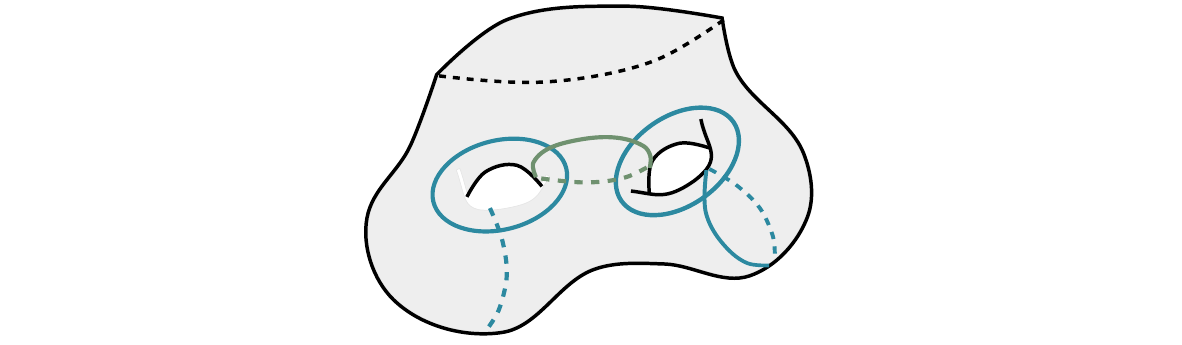
\caption{Changing the type of the indices $i$ and $j$ by twisting along $c$.}
\label{fig:spin}
\end{figure}

The \textbf{Arf invariant} $Arf(q) \in \Z/2\Z$ of $q$ can be defined as follows
\[Arf(q):= \# \,  \{ i \; \vert \; q(a_i)  \cdot q(b_i)=1  \} \, \, mod \, 2. \]
The following theorem due to C. Arf classifies all quadratic forms on $\Sigma$, see \cite{Arf} or \cite[Section 11.4]{scorpan} for a modern treatment.

\begin{theorem}\label{thm:arf}
For two quadratic forms $q, \, q' :  \scc(\Sigma) \rightarrow \Z/2\Z$, there exists $\phi \in \MCG(\Sigma)$ such that $q'=q\circ \phi$ if and only if $\Arf(q)=\Arf(q')$.
\end{theorem}

\subsection{The spin symplectic group}\label{sec:spinsymp}

The aim of this section is to show the following.

\begin{Proposition}\label{prop:spqtq}
Assume that $g \geq 5$. Then we have $\spaut(H_1(\Sigma,\Z),q)<[\tq]$.
\end{Proposition}

\begin{proof}
For the sake of simplicity, let us denote by $\spaut(q):=\spaut(H_1(\Sigma,\Z),q)$ and by $\spaut(q,\Z/2\Z)$ the group of automorphisms of $H_1(\Sigma,\Z/2\Z)$ which preserve both the intersection product and the form $q$. We have the short exact sequence of groups
\[
1\rightarrow \spaut^{(2)}\rightarrow \spaut(q)\rightarrow \spaut(q,\Z/2\Z)\rightarrow 1
\]
where $\spaut^{(2)}$ is the subgroup of $\spaut(q)$ consisting of elements whose matrix in any symplectic basis is the identity modulo $2$. Hence, it suffices to show that $\spaut^{(2)} < [\tq]$ and $\spaut(q,\Z/2\Z) < [\tq]_2$.

By  \cite[I.5.1]{chevalley}, the group $\spaut(q,\Z/2\Z)$ is generated by the transvections along the elements of the set $q^{-1}(1)\subset H_1(\Sigma,\Z/2\Z)$ for $g\geq 3$. By definition, any such transvection is the reduction $mod \: 2$ of an admissible twist. It follows that $\spaut(q,\Z/2\Z) < [\tq]_2$.

To show that $\spaut^{(2)} < [\tq]$, recall that $\spaut^{(2)}$ is generated by squares of transvections (see\cite[Lemma 5]{John85}). The result follows from the lemma below.
\end{proof}

\begin{Lemma}\label{lem:square}
Assume that $g\geq5$. For any non-separating $d \in \scc(\Sigma)$, we have $\tau_d^2 \in \tq$.
\end{Lemma}

In order to prove the above lemma, we will use the genus-$2$ star relation as given in \cite[Proposition 4.5]{Salter1}. Considering the simple closed curves of Figure \ref{fig:starrel}, the genus-$2$ star relation can be stated as follows
\[ \tau_d^2= (\tau_{a_1}\tau_{a_1'}\tau_{c_1} ...\tau_{c_4})^5(\tau_{a_3}\tau_{a_3'})^{-1}. \]
\begin{figure}[h]
\centering
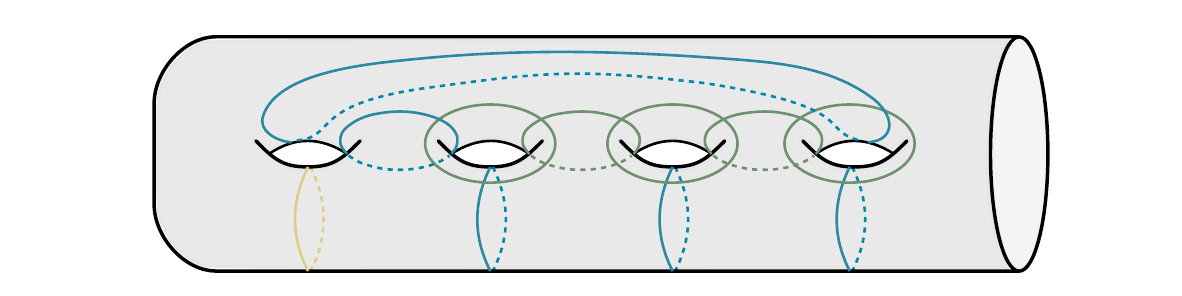
\caption{The curves involved in the genus-$2$ star relation.}
\label{fig:starrel}
\end{figure}

\begin{proof}[Proof of Lemma \ref{lem:square}]
If $q(d)=1$, then $\tau_d$ is admissible and there is nothing to prove. Let us assume that $q(d)=0$.
We claim that there exist admissible curves $a_1,a_1',c_1,...c_4,a_3,a_3' \in \scc(\Sigma)$ arranged as in Figure \ref{fig:starrel}. Applying the genus-$2$ star relation, we conclude that $\tau_d^2 \in \tq$.

Let us prove the claim. As the surface $\Sigma$ has genus $g\geq 5$, Theorem \ref{thm:arf} implies that we can always find a $q$-symplectic basis $a_1,b_1,...,a_g,b_g$ such that 
\[ q(a_1)  \cdot q(b_1)=q(a_3)  \cdot q(b_3)=1 \text{ and } q(a_2)  \cdot q(b_2)=q(a_4) \cdot q(b_4)=0\]
with $a_4=d$. Indeed,  we can always correct the Arf invariant on the rest of the $q$-symplectic basis, see Section \ref{sec:spin}. In particular, we have $q(b_2)=1$ according to Definition \ref{def:qsymp}.

Now, define $c_1:=b_1$, $c_3:=b_2$ and $c_5:=b_3$. It follows from the construction that  the curves $a_1$, $a_3$, $c_1$, $c_3$ and $c_5$ are admissible. By the change-of-coordinate principle \cite[Section 1.3]{farbmarg}, it is always possible to complete the latter collection of curves with curves  $a_1'$, $c_2$, $c_4$ and $a_3'$ in order to achieve the configuration of Figure \ref{fig:starrel}. It remains to show that $a_1'$, $c_2$, $c_4$ and $a_3'$ are also admissible. This follows from the computations
\[
\begin{array}{rl}
q(a_1') & =q(a_1)+q(d)=1+0=1,\\
q(c_2) & =q(a_1)+q(a_2)=1+0=1,\\
q(c_4) & =q(a_2)+q(a_3)=0+1=1,\\
q(a_3') & =q(a_1')+q(c_2)+q(c_4)=1+1+1=1.
\end{array}
\]
\end{proof}

\subsection{The Torelli group}\label{sec:torelli}

Recall that the Torelli group $\T < \MCG(\Sigma)$ is defined as the kernel of the morphism
\eqref{eq:natmorph}. The main result of this section is the following.

\begin{Proposition}\label{prop:torelli}
For $g \geq 5$, we have $\T <\tq$.
\end{Proposition}

To show this, we will use a set of generators of $\T$ provided by Johnson in \cite[Theorem 2]{John79}. A \textbf{bounding pair} $(\alpha,\beta)$ is a pair of  disjoint, non-separating curves $\alpha, \beta \in \scc(\Sigma)$ that are homologous to each other. The \textbf{genus} of a bounding pair is by definition the genus of the smaller of the two subsurfaces bounded by the pair. For a bounding pair $(\alpha,\beta)$, define the \textbf{bounding pair map}
\[\tau_{(\alpha,\beta)}:= \tau_\alpha \tau_\beta^{-1} \in \MCG(\Sigma). \]
Define the genus of a bounding pair map to be the genus of its underlying bounding pair.

\begin{theorem}[\cite{John79}]\label{thm:johnson}
For $g \geq 3$, the Torelli group $\T$ is generated by the bounding pair maps of genus $1$.
\end{theorem}

Now, the strategy is to show that any bounding pair map of genus $1$ is an element of $\tq$. To that aim, we will need a particular instance of the chain relation (see \cite[Proposition 4.12]{farbmarg}). 

Let $a,b,c \in \scc(\Sigma)$ such that $i(a,b)=i(b,c)=1$ and $i(a,c)=0$ where $i(.,.)$ denotes the algebraic intersection number. Denote by $\alpha$ and  $\beta$ the two boundary components of the tubular neighbourhood of $a \cup b \cup c$ in $\Sigma$. Then the chain relation states that 
\begin{equation}\label{eq:chrel}
(\tau_a \tau_b \tau_c)^4= \tau_\alpha \tau_\beta. 
\end{equation}
In the present case, the element $\tau_a$ commutes with $\tau_c$ and $\tau_b$ commutes with both $\tau_\alpha$ nand $ \tau_\beta$. Using these and the braid relation \cite[Proposition 3.11]{farbmarg}, we derive the following relation
\begin{equation}\label{eq:chrel2}
(\tau_b^2 \tau_a \tau_b^2 \tau_c)^2= \tau_\alpha \tau_\beta. 
\end{equation}
The proof goes as follows
\[
\begin{array}{rl}
\tau_\alpha \tau_\beta & = \tau_b (\tau_\alpha \tau_\beta) \tau_b^{-1} = \tau_b (\tau_a \tau_b \tau_c)^4 \tau_b^{-1}\\
								& = \big( \tau_b (\tau_a \tau_b \tau_c \tau_a \tau_b \tau_c) \tau_b^{-1}\big)^2\\
								&= \big( \tau_b (\tau_a \tau_b \tau_a \tau_c \tau_b \tau_c) \tau_b^{-1}\big)^2\\
								&= \big( \tau_b (\tau_b \tau_a \tau_b \tau_b \tau_c \tau_b) \tau_b^{-1}\big)^2\\
								&=(\tau_b^2 \tau_a \tau_b^2 \tau_c)^2.
\end{array}
\]
We have now all the ingredients we need.
 
\begin{proof}[Proof of Proposition \ref{prop:torelli}]
As announced above, we will prove that any bounding pair map of genus $1$ is an element of $\tq$ and conclude with Theorem \ref{thm:johnson}.

Let $(\alpha,\beta)$ be a bounding pair of genus $1$ and let us show that $\tau_{(\alpha,\beta)}$ is an element of $\tq$. As $\alpha$ and $\beta$ are homologous, we have $q(\alpha)=q(\beta)$. If $q(\alpha)=1$, then both $\alpha$ and $\beta$ are admissible curves and there is nothing to prove.

Let us assume that $q(\alpha)=0$ and denote by $\Sigma' \subset \Sigma$ the subsurface of genus $1$ bounded by $\alpha$ and $\beta$. From Section \ref{sec:spin}, we know that we can always find simple closed curves $a$ and $b$ in $\Sigma'$ such that $i(a,b)=1$ and $q(a)=1$, up to exchanging $a$ with $b$. By the change-of-coordinate principle \cite[Section 1.3]{farbmarg}, we can find a curve $c \subset \Sigma'$ as pictured in Figure \ref{fig:bp}. 

On the one hand, we have that $\tau_a$, $\tau_b^2$ and $\tau_c$ are elements of $\tq$. Indeed, we have that $q(c) = q(a) + q(\alpha)=1$ so that $\tau_c$ is admissible and $\tau_b^2 \in \tq$ by Lemma \ref{lem:square}. 

On the other hand, the curves $a$, $b$, $c$, $\alpha$ and $\beta$ are in the configuration of the chain relation \eqref{eq:chrel}. In particular, we are in position to apply the relation \eqref{eq:chrel2} and deduce that $\tau_\alpha \tau_\beta$ is in $\tq$. Applying Lemma \ref{lem:square} again, we have that $\tau_\beta^2 \in \tq$ and in turn that $\tau_{(\alpha,\beta)} = \tau_\alpha \tau_\beta \tau_\beta^{-2}$ is in $\tq$.
\end{proof}

\begin{figure}[h]
\centering
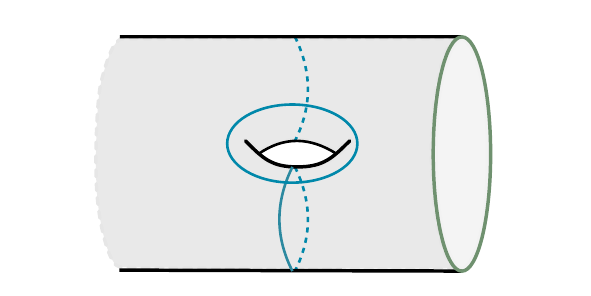
\caption{A bounding pair map of genus $1$.}
\label{fig:bp}
\end{figure}

\section{The spin case}\label{sec:spinst}

In this section, we come back to the discussion on Theorems \ref{thm:ispin} and \ref{thm:spinm}. Thus, we assume that the image of $\Xd\rightarrow \LK^*$ is of dimension~$2$ and that the largest order $n$ of a root $\LS$ of $K_{\Xd}\otimes \Li$ is even. In particular, the polygon $\da$ is even, as we recalled in Section \ref{sec:poly}.

Notice also that since the Picard group of $\Xd$ is free \cite[Proposition p.63]{Ful}, the $n\textsuperscript{th}$ root $\LS$ of $K_{\Xd} \otimes \Li$ is unique.

Let us fix a smooth curve $C_0$ in $\LL$. Since $n$ is even, the restriction of $\LS^{\otimes \frac{n}{2}}$ to $C_0$ is a square root of $K_{C_0}$, $i.e.$ a spin structure on $C_0$ (see \cite[Proposition 2.7]{CL}). In this section, we will denote this spin structure by $S$ and its associated quadratic form by $q$ (see Section \ref{sec:generators}).

Recall the notion of admissible twist and the group $\tq$ generated by them, see Definition \ref{def:adm}. We are now in position to prove Theorem \ref{thm:spinm}.

\begin{proof}[Proof of Theorem \ref{thm:spinm}.]
By Theorem B and Section 10 in \cite{Salter2}, we have that $\tq$ is a subgroup of $\im(\mu)$. Also, the line bundle $\Li$ corresponds to a lattice polygon $\Delta$ such that $\da$ is even, see Section  \ref{sec:poly}. It implies that the number of lattice points of $\da$, that is also the genus $g_\Li$ of $C_0$ is at least $6$. In particular, we can apply all the results of Section \ref{sec:generators}.

If $n=2$, Theorem \ref{thm:spin} implies that $\MCG(C_0,q) < \im(\mu)$ which has to be an equality by \cite[Proposition 2.7]{CL}.

For $n$ even, Proposition \ref{prop:spqtq} implies that $\spaut(H_1(C_0,\Z),q) < \im([\mu])$ which has to be an equality by \cite[Proposition 2.7]{CL}.
\end{proof}

We now conclude this section with the proof of Theorem \ref{thm:ispin}. First, notice that for any simple closed curve $\delta$ on $C_0$, the condition that its tangent framing lifts to the spin structure $S$ is equivalent to the equality $q(\delta) = 1$.

 On the other hand, we know from Theorem \ref{thm:rappel} that any $A$-cycle on $C_0$ is a vanishing cycle. Moreover, it follows from \cite[Proposition 3.23]{Voisin2} that all the other vanishing cycles can be obtained from each other under the action of $\im(\mu) = \MCG(C_0,S)$. Thus, it is enough to show that if $\delta\subset C_0$ is a simple closed curve with $q(\delta) = 1$, then there exists an element $\phi\in\MCG(C_0,S)$ such that $\phi(\delta)$ is an $A$-cycle. This is what is proven in the following lemma.

\begin{Lemma}\label{lem:transitive}
  Let $C$ be a closed oriented genus $g$ surface with $g\geq 1$ and let $q : H_1(C,\Z/2\Z)\rightarrow \Z/2 \Z$ be a quadratic form on $C$. \\
For any two non-separating simple closed curves $a$ and $a'$ on $C$, $q(a) = q(a')$ if and only if there exists $\varphi\in \MCG(C,q)$ such that $\varphi(a)$ is isotopic to $a'$.
\end{Lemma}

\begin{proof}
Let $a$ and $a'$ be two non-separating simple closed curves on $C$ and assume that $q(a) = q(a')$. Cut $C$ along $a$ on one hand and along $a'$ on the other hand to obtain two surfaces $C_1$ and $C_2$ with boundary. Moreover, the spin structure on $C$ induces spin structures $S_1$ and $S_2$ on $C_1$ and $C_2$ with prescribed boundary conditions.

Choose an orientation preserving diffeomorphism $\phi : (C_1,\partial C_1) \rightarrow (C_2,\partial C_2)$. The pullback $\phi^*S_2$ is a spin structure on $C_1$ with the same Arf invariant and with the same boundary condition as $S_1$ since $q(a) = q(a')$. It follows from \cite[Theorem 2.9]{RW} that there exists $\psi \in\MCG(C_1,\partial C_1)$ such that $\psi^*(\phi^*S_2) = S_1$.

 Gluing back the surfaces $C_1$ and $C_2$ using \cite[Corollary 2.12]{RW}, $\phi\circ\psi$ induces an element $\varphi\in\MCG(C,q)$ such that $\varphi(a)=a'$.
\end{proof}

\section{Appendix: computation of the canonical spin structure}\label{sec:app}

In this section, we work under the assumptions of Theorem  \ref{thm:spinm}. In such a case, Proposition \ref{prop:rappel} allows us to compute the value of the quadratic form $q$ associated to the spin structure $S:=\LS^{\otimes \frac{n}{2}}$ on a symplectic basis of $H_1(C_0,\Z/2\Z)$. Indeed,
\begin{itemize}
\item if $v$ is any lattice point of $\da$, then $[\tau_{v}]\in\im(\mu)\subset \spaut(H_1(C_0,\Z),S)$ (see Theorem \ref{thm:rappel}). Necessarily, we have $q([\delta_{v}])=1$ by Lemma \ref{lem:adm}.
\item if $v$ is an even lattice point, we can connect it to the boundary of $\Delta$ using a sequence $\sigma_1,\ldots, \sigma_{2k+1}$ of primitive integer segments such that the union of $\sigma_1,\ldots,\sigma_{2k}$ is a segment joining $v$ to a vertex of $\da$ and $\sigma_{2k+1}$ is a bridge ending at this vertex. We know from Proposition \ref{prop:rappel} and Theorem \ref{thm:rappel} that for any $i=1,\ldots,2k+1$, $[\tau_{\sigma_i}]\in\im(\mmu)\subset \spaut(H_1(C_0,\Z),S)$. In particular, $q([\delta_{\sigma_i}])=1$ and $q([\delta_{\sigma_1}]+\ldots +[\delta_{\sigma_{2k+1}}])=1$.
\item If $v$ is not even, take a primitive integer segment $\sigma'$ joining $v$ to an even lattice point $w$ and take $\sigma_1,\ldots, \sigma_{2k+1}$ to be a sequence of primitive integer vectors joining $w$ to the boundary of $\Delta$ as in the previous point. Then it follows again that $q([\delta_{\sigma'}]+[\delta_{\sigma_1}]+\ldots+[\delta_{\sigma_{2k+1}}]) = 0$.
\end{itemize}

The above computation requires that we know that certain Dehn twists are in the image of the monodromy. In general, it is quite hard to determine whether this condition is satisfied for a particular Dehn twists.

We now give a more direct proof of this computation  without using Proposition \ref{prop:rappel}. Recall first that the quadratic form $q$ can be described in the following way. 

For a simple closed curve $\delta$ on $C_0$, choose an embedded orientable surface $D$ in $\Xd$ which is transverse to $C_0$ except along its boundary $\delta$ where it meets $C_0$ normally. Trivialize the normal to $D$ in $\Xd$ and denote by $e\in \Z/2\Z$ the index modulo $2$ of the normal to $\delta$ in $C_0$ in this trivialization. Define $q(\delta) = e + \vert \itr(D) \cap C_0 \vert \in \Z/2\Z$. The condition that $K_{\Xd} \otimes \Li$ admits a square root implies that $C_0$ is a characteristic surface in $\Xd$ and thus that $q(\delta)$ does not depend on the choice of $D$. Moreover, the function $q:H_1(C_0,\Z/2\Z) \rightarrow \Z/2\Z$ is a quadratic form as defined in Section \ref{sec:generators} (see \cite[Corollary p.514]{scorpan}).

\begin{Definition}\label{def:bcycle}
For any $v \in \da \cap \Z^2$, a \textbf{$B$-cycle} at $v$ is a simple closed curve in $C_0$ globally invariant by complex conjugation and such that the corresponding path in $\Delta$ joins $v$ to the boundary of $\Delta$.
\end{Definition}

Notice that two $B$-cycles at $v$ are always homologous if one orients them carefully. In particular, the $mod \: 2$ homology class of the $B$-cycle at $v$ is well defined. the transvection along a $B$-cycle at $v$ does not depend on a choice of $B$-cycle. Denote this class by $\left[ \delta \right] \in H_1(C_0, \Z/2\Z)$.

\begin{Proposition}\label{prop:q}
Fix $v \in \da \cap \Z^2$. Then, we have 
\begin{enumerate}
\item $q(\left[ \delta \right])=1$ if and only if $v$ is an even lattice point of $\da$,
\item $q(\left[ \delta_v \right])=1$.
\end{enumerate}
\end{Proposition}

In particular, the latter proposition allows one to compute the quadratic form $q$ on a symplectic basis of $H_1(C_0, \Z)$. Indeed, number the lattice points of $\da$ from $v_1$ to $v_{g_{\Li}}$ and let us consider the elements $a_1,b_1,\ldots,a_{g_{\Li}},b_{g_{\Li}}$ of $H_1(C_0,\Z)$ where $a_i$ is the homology class of $\delta_{v_i}$ oriented arbitrarily, and $b_i$ is the homology class of a $B$-cycle at $v_i$. It follows from Remark \ref{rem:sympbas} that $a_1,b_1,\ldots,a_{g_{\Li}},b_{g_{\Li}}$ is a symplectic basis of $H_1(C_0,\Z)$ once the $B$-cycles are oriented carefully.

In order to prove Proposition \ref{prop:q}, we will need the following. 
 
\begin{Lemma}\label{lem:qc}
Let $\sigma\subset \Delta$ be a primitive integer segment not contained in $\partial \Delta$ and such that $\sigma$ is parallel to an edge $\epsilon \in \da$. Then $q(\left[\delta_\sigma\right]) =1$ if and only if one extremity of $\sigma$ is even.
\end{Lemma}

In order to compute $q$, we will construct a membrane $D$ using some consideration on the amoeba of $C_0$. 

Recall that the amoeba map $\A : (\C^\ast)^2 \rightarrow \R^2$ is defined as $\A(z,w):=\big(\log\vert z \vert, \log \vert w \vert \big)$ and that the amoeba of $C_0$ is simply the subset $\A(C_0) \subset \R^2$. For simple Harnack curves, $\A$ is a $2$-to-$1$ covering folded along $\R C_0$, identifying $\pi_0(\R C \cap (\C^\ast)^2)$ with $\pi_0\big(\partial \A(C_0)\big)$. Moreover, the order map $\nu : \pi_0\big( \R^2 \setminus \A(C_0)\big) \rightarrow \Delta \cap \Z^2$ introduced in \cite{FPT} is a bijection, see \cite{Mikh}. 

Before proving Lemma \ref{lem:qc}, recall at last the normalization procedure of \cite{CL}.

\begin{Definition}\label{def:normalization}
If $\kappa$ is a vertex of $\da$, a \textbf{normalization} of $\Delta$ at $\kappa$ is an invertible affine transformation $A:\R^2\rightarrow\R^2$ preserving the lattice and mapping $\kappa$ to $(0,0)$ and its adjacent edges in $\da$ to the edges directed by $(1,0)$ and $(0,1)$.
\end{Definition}

We will usually abuse notation and keep calling $\Delta$ the image of $\Delta$ under a normalization. 

There are two possibilities for the polygon $\Delta$ near $\kappa$, as depicted in the Figure \ref{fig:cornercase}.
Indeed, since the normal fan of $\Delta$ is a refinement of that of $\da$, there is an edge of $\Delta$ supported on the line $y=-1$ and one supported on the line $x=-1$. The two cases are as follows :
\begin{itemize}
\item either those two edges meet, see the left case of Figure \ref{fig:cornercase};
\item or there is one edge $\epsilon$ of integer length $1$ and slope $(-1,1)$ between those two, see the right case of Figure \ref{fig:cornercase}.
\end{itemize}

\begin{figure}[h]
\centering
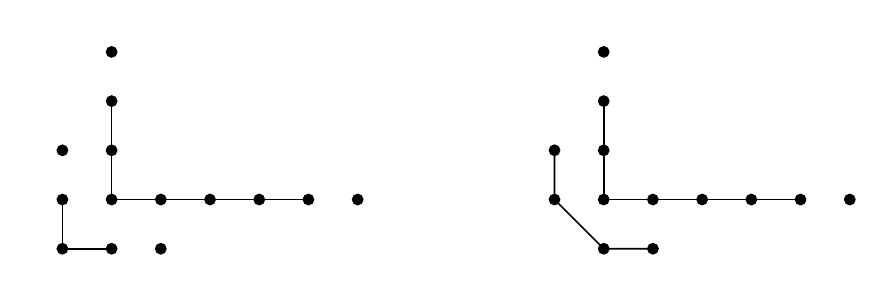
\caption{The polygon $\Delta$ near $\kappa$} \label{fig:cornercase}
\end{figure}

\begin{proof}[Proof of Lemma \ref{lem:qc}]
After normalizing $\Delta$ at one of the vertices $\kappa$ of $\epsilon$, we can assume without any loss of generality that $\sigma$ is the primitive integer segment of slope $(1,0)$ starting at some point $(i,j)$. In particular, $\sigma$ has an even extremity if and only if $j$ is even.

As a consequence of \cite[Proposition 4.6]{CL}, the value of $q(\left[\delta_\sigma\right])$ does not depend on the particular simple Harnack curve $C_0$ that we pick. Therefore, we can consider a curve $C_0$ for which there exists a horizontal segment $s \subset \A(C_0)$ joining the boundary component of order $(i,j)$ to the one of order $(i+1,j)$. One can construct such a curve using patchworking according to a subdivision of $\Delta$ containing $\sigma$, see \cite[Section 3]{BIMS}.

We now construct a membrane $D\subset X$ intersecting $C_0$ normally along $\delta := \A^{-1}(s)\cap C_0$. Observe that $\delta$ is a simple closed curve isotopic to $\delta_\sigma$, in particular $q(\left[\delta_\sigma\right])=q(\left[\delta\right])$.  Denote by $p$ the rightmost point of $s$ and by  $\lambda$  the geodesic of slope $(0,1)$ in the argument torus $\A^{-1}(p) \simeq S^1\times S^1$. By \cite[Lemma 11]{Mikh}, the loop $\delta$ has class $\left[ (0,1) \right] \in H_1 ( \ttor, \Z)$. It follows that we can construct an isotopy from $\delta$ to $\lambda$ inside $\A^{-1}(s)$, see Figure \ref{fig:membrane}. Denote by $I$ the total space of this isotopy. Take now $\varepsilon>0$ small enough such that $\left]p, p+(\varepsilon,0)\right]$ is contained in the complement component of $\A(C_0)$ of order $(i+1,j)$. Consider a convex smoothing $F$ of the piecewise linear curve $\left[ p, p+(\varepsilon, 0) \right] \cup \left\lbrace p+(\varepsilon, y) \, \vert \,  y\leq 0\right\rbrace$ in a small neighbourhood of $p+(\varepsilon, 0)$ disjoint from $\A(C_0)$. Define $D$ to be the closure of $I \cup (F \times \lambda)$ in $ \Xd$. Notice that the part of $D$ supported on the half-line $\left\lbrace p+(\epsilon, y) \, \vert \,  y \leq 0 \right\rbrace$ is a holomorphic annulus contained in $\left\lbrace (z,w) \in (\C^\ast)^2 \,  \vert \,  z = c \right\rbrace$ for some $c\in \C^\ast$. The closure of the cylinder in $\Xd$ hits the toric divisor $D_\epsilon$ in a point and is therefore a holomorphic disc. It follows that $D \subset \Xd$ is a smooth membrane intersecting  $C_0$ normally along $\delta$ and transversally anywhere else, for a generic $\varepsilon$.

In order to compute $q(\left[\delta\right])$, we need to determine the index $e \in \Z/2\Z$ of the normal of $\delta \subset C_0$ with respect  to a trivialization of the normal of $D$. To that end, notice that the vector  $(1,1)$ is never tangent to $\A(D)$ and then provides a trivialization of the normal of $D$ inside the chart $\C^\ast \times \C \subset \Xd$, where $D_\epsilon = \C^\ast \times  \{0\}$. Now, we can choose the normal vector field of $\delta$ inside $C_0$ so that its image by the tangent map $T \A$ has a strictly positive vertical coordinate. It is then isotopic to the restriction along $s$ of the normal of $D$ chosen above. We deduce that the Euler index $e$ is trivial and that $q(\left[\delta\right]) = \vert \itr (D) \cap C_0 \vert \; mod\, 2$. 

It remains to compute the number of  intersection points  $\vert \itr (D) \cap C_0 \vert$. As this intersection is contained in the cylinder $\left\lbrace (z,w) \in (\C^\ast)^2 \vert \,  z = c \right\rbrace$, it follows from \cite[Lemma 2.2]{FPT} that $\vert \itr (D) \cap C_0 \vert$ is given by $j-j_{min}$ where $j_{min}:=\min \{ \tilde{j} \, \vert \, (\tilde{i}, \tilde{j}) \in \Delta \cap \Z^2 \}$. From the way we normalized $\Delta$, we deduce that $j_{min}=-1$ and in turn that $q(\left[\delta\right]) = j+1 \; mod\, 2$. Therefore $q(\left[\delta\right]) =1$ if and only if $\sigma$ has an even extremity.
\end{proof}

\begin{figure}[h]
\centering
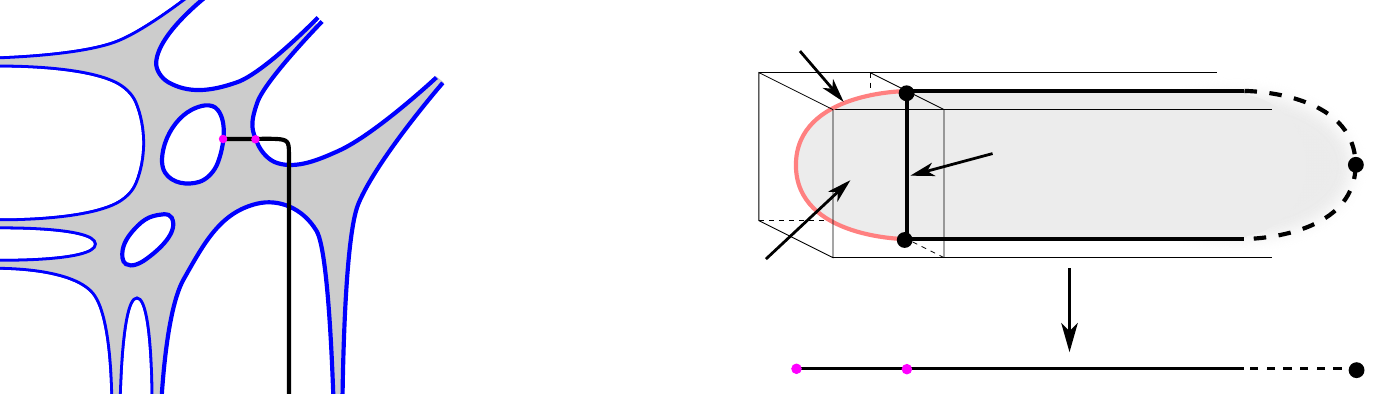
\caption{The construction of $D$. On the left, the position of $s \cup F$ with respect to $\A(C_0)$. On the right, $D$ (in grey) inside the solid torus $\overline{\A^{-1}(s \cup F)}$.}
\label{fig:membrane}
\end{figure}

\begin{proof}[Proof of Proposition \ref{prop:q}]
For the first assertion, notice that we can always normalize $\Delta$ in such a way that $v$ can be joined to $(-1,0) \in \partial \Delta$ by a concatenation of pairwise distinct primitive integer segments that are either horizontal or vertical. Indeed, it suffice to normalize as done in Figure \ref{fig:cornercase}.

Consider such collection $\sigma_0$, ..., $\sigma_k \subset \Delta$ of primitive integer segments. As the $B$-cycle $[\delta]$ at $v$ is given by $\left[ \delta_{\sigma_0}\right] + \dots + \left[\delta_{\sigma_k} \right] \in H_1(C_0,\Z)$ and as $\ip{\delta_{\sigma_l}}{\delta_{\sigma_m}} =0$ for any couple $(l,m)$, we have $$q(\left[ \delta_{\sigma_0}\right] + \dots + \left[ \delta_{\sigma_k} \right]) = q(\left[ \delta_{\sigma_0}\right]) + \dots  +  q(\left[ \delta_{\sigma_k} \right]).$$
By Lemma \ref{lem:qc}, we have that $q(\left[ \delta_{\sigma_{2j}}\right])= q(\left[ \delta_{\sigma_{2j+1}}\right])$ for any integer $j\leq (k-1)/2$. It follows that $q(\left[ \delta \right])=1$ if and only if $k$ is even and $q(\left[ \delta_{\sigma_k} \right])=1$, $i.e$ $v$ is an even lattice point of $\da$.

For the second assertion, consider the closed disc $D\subset (\R^\ast)^2$ bounded by the $A$-cycle $\delta_v$. By looking at the amoebas of $D$ and $C_0$, we deduce that $D$ and $C_0$ only intersect along $\delta_v$. As $D \subset (\R^\ast)^2$, the constant vector field $(i,i)$ on $\D$ provides a trivialization of its normal in $\Xd$. Multiplying the tangent vector field of $\delta_v$ by $i$ provides the normal vector field of $\delta_v$ in $C_0$. As the rotational index of $\delta_v \subset (\R^\ast)^2$  is $1$, it follows that the Euler index of $D$ with respect to $C_0$ is $1$. We conclude that $q(\left[ \delta_v \right])=1$.
\end{proof}

\bibliographystyle{alpha}

\vspace{1cm}
Email: remicretois@yahoo.fr, lang@math.su.se

\end{document}